\documentclass[reqno]{amsart}

\usepackage{hyperref}
\usepackage{amsmath,amsthm,amssymb,amscd}

\newtheorem{thm}{Theorem}
\newtheorem{prop}[thm]{Proposition}
\newtheorem{lem}[thm]{Lemma}

\theoremstyle{definition}
\newtheorem{dfn}[thm]{Definition}
\newtheorem{rem}[thm]{Remark}
\newtheorem{prob}[thm]{Problem}
\newtheorem{conj}[thm]{Conjecture}

\numberwithin{thm}{section}
\numberwithin{equation}{section}

\newcommand{\step}[1]{\medskip\noindent\textit{#1.}}

\title[Euler characteristics]{Euler characteristics of\\collapsing Alexandrov spaces}
\author[T. Fujioka]{Tadashi Fujioka}
\address{Department of Mathematics, Osaka University, Toyonaka, Osaka 560-0043, Japan}
\email{fujioka@cr.math.sci.osaka-u.ac.jp, tfujioka210@gmail.com}
\date{\today}
\thanks{Supported by JSPS KAKENHI Grant Number 22KJ2099 (22J00100)}

\begin{document}

\begin{abstract}
We prove that the Euler characteristic of a collapsing Alexandrov space (in particular, a Riemannian manifold) is equal to the sum of the products of the Euler characteristics with compact support of the strata of the limit space and the Euler characteristics of the fibers over the strata.
This was conjectured by Semyon Alesker.
\end{abstract}

\maketitle

\section{Introduction}\label{sec:intro}

In this paper we consider the following problem.
We fix an upper bound $n$ for dimension and a lower bound $\kappa$ for sectional curvature.

\begin{prob}\label{prob:main}
Let $X$ be a $k$-dimensional compact Alexandrov space, where $k\le n$.
Suppose $\mu=\mu(X)>0$ is small enough and let $M$ be an $n$-dimensional Alexandrov space (in particular, a Riemannian manifold) that is $\mu$-close to $X$ with respect to the Gromov-Hausdorff distance.
Describe the topology of $M$ in terms of the geometry of $X$.
\end{prob}

An Alexandrov space is a metric space with a lower sectional curvature bound, introduced by Burago-Gromov-Perelman \cite{BGP}.
The motivation for the above problem stems from the well-known fact that the family of $n$-dimensional Riemannian manifolds with sectional curvature $\ge\kappa$ and diameter $\le D$ is precompact in the Gromov-Hausdorff topology.
The limit objects are Alexandrov spaces with curvature $\ge\kappa$ and dimension $\le n$.
More generally, the family of Alexandrov spaces with curvature $\ge\kappa$, diameter $\le D$, and dimension $\le n$ is compact in the Gromov-Hausdorff topology.
Therefore, in principle, the solution to the above problem allows us to understand the topology of spaces in these (pre)compact families by covering them by a finite number of neighborhoods of limit spaces.

In the case $k=n$, Perelman's stability theorem \cite{Per:alex} (cf.\ \cite{K:stab}) solves the above problem completely, that is, $M$ is homeomorphic to $X$.
Hence we will consider the case $k<n$, called \textit{collapse}.
Although collapsing in low dimensions has been well studied (\cite{SY:3dim}, \cite{Y:4dim}, \cite{MY:3dim}, \cite{AKP'}), so far there is no theory in general dimensions.
The few exceptions are special cases when $X$ satisfies some regularity conditions (\cite{Y:col}, \cite{Per:col}), where some fibration structures of $M$ over $X$ are obtained.
In general, it is expected that $M$ admits some singular fibration structure over $X$, where the singular fibers arise over the singular strata of $X$.
See also \cite{CFG} and the references therein for collapsing under two-sided bounds on sectional curvature.

Recently Alesker \cite{A:conj} proposed some conjectures on collapsing Riemannian manifolds and Alexandrov spaces in terms of their intrinsic volumes (also known as Lipschitz-Killing curvatures).
For a closed $n$-dimensional Riemannian manifold $M$, the \textit{$i$-th intrinsic volume}, denoted by $V_i(M)$ ($0\le i\le n$), is a geometric quantity defined as follows (see \cite{A:conj} for the precise definition and references).
We first embed $M$ isometrically into Euclidean space (by the Nash embedding, for example) and consider the volume of its $\varepsilon$-neighborhood.
Then it is a polynomial in $\varepsilon\ll1$, and its coefficients, after appropriate normalization, turn out to be independent of the embedding, which are called the intrinsic volumes of $M$.
In fact they can be defined directly in terms of integrals of the Riemann curvature tensor of $M$.
For example, $V_n(M)$ is the volume of $M$, $V_0(M)$ is the Euler characteristic of $M$, and $V_{n-2}(M)$ is proportional to the integral of the scalar curvature of $M$ (and $V_i(M)$ vanishes if $n-i$ is odd).

The following is a brief summary of part of Alesker's conjecture.

\begin{conj}[Alesker \cite{A:conj}]\label{conj:A}
Let $M_j$ be a sequence of $n$-dimensional Riemannian manifolds with sectional curvature $\ge\kappa$.
Suppose $M_j$ converges to a compact Alexandrov space $X$.
\begin{enumerate}
\item If $M_j$ does not collapse, then $\lim_{j\to\infty}V_i(M_j)$ exists.
\item If $M_j$ collapses, then
\begin{enumerate}
\item there is a subsequence such that $\lim_{j\to\infty}V_i(M_j)$ exists;
\item the limit value will be written as
\[\sum_{E\in\mathcal E}F(E)\cdot V_i(E),\]
where $\mathcal E$ denotes the set of the strata of $X$, $F$ is an integer valued function defined on $\mathcal E$, and $V_i(E)$ is the ``$i$-th intrinsic volume'' of $E$.
\end{enumerate}
\end{enumerate}
\end{conj}

Here the strata of $X$ mean (the main parts of) primitive extremal subsets in the sense of Perelman-Petrunin \cite{PP:ext}.
This stratification reflects both the geometric and topological structures of Alexandrov spaces, and is very closely related to the collapsing phenomena.
Note that the number of such strata is finite and that the above quantity $V_i(E)$ has not yet been defined.
It also should be mentioned that the existence of the expected function $F$ was stated in \cite[Theorem 4.6]{A:conj} as an unpublished result of Petrunin.
See \cite{A:conj} for more details and further conjectures.

At present only the following results are known for Conjecture \ref{conj:A}.
\begin{itemize}
\item For $V_n$, the volume: the claim follows from the volume convergence theorem of Burago-Gromov-Perelman \cite[Theorem 10.8]{BGP}.
\item For $V_0$, the Euler characteristic: (1) follows from Perelman's stability theorem \cite{Per:alex} (cf.\ \cite{K:stab}) and (2a) follows from Gromov's Betti number theorem \cite{G:bet} (cf.\ \cite{LS}, \cite{Y:ess}).
\item For $V_{n-2}$, the total scalar curvature: (1) was recently proved by Lebedeva-Petrunin \cite[Subcorollary 1.3]{LP} and (2a) was proved by Petrunin \cite{Pet:int}.
\item The special case of a Riemannian submersion was verified by Alesker \cite{A:conv}.
\item The $2$-dimensional case was verified by Alesker-Katz-Prosanov \cite{AKP'}.
\end{itemize}

In this paper we prove (2b) of Conjecture \ref{conj:A} for the $0$-th intrinsic volume, i.e., the Euler characteristic, not only for Riemannian manifolds but also for Alexandrov spaces.
We denote by $\chi$ the Euler characteristic and by $\chi_c$ the Euler characteristic with compact support, where the coefficient field of singular cohomology is fixed and omitted.

\begin{thm}\label{thm:main}
Let $X$ and $M$ be Alexandrov spaces as in Problem \ref{prob:main}.
Then
\[\chi(M)=\sum_{E\in\mathcal E}\chi_c(\mathring E)\cdot\chi(F_E),\]
where $\mathcal E$ denotes the set of primitive extremal subsets of $X$, $\mathring E$ is the main part of $E$, and $F_E$ is a regular fiber over $E$ in $M$ (see Section \ref{sec:pre} for the definitions).
\end{thm}

For a regular (= nonsingular) fibration, the Euler characteristic of the total space splits into the product of those of the fiber and the base space.
Therefore, the above formula can be interpreted as capturing the expected singular fibration structure of $M$ over $X$, at least at the level of the Euler characteristic.

\begin{rem}
If $k=n$, then all the regular fibers are contractible.
This follows from the parametrized version of Perelman's stability theorem \cite[Theorem 4.3]{Per:alex} (cf. \cite[Theorem 7.8]{K:stab}).
In particular we have $\chi(X)=\sum_{E\in\mathcal E}\chi_c(\mathring E)$.
\end{rem}

\begin{rem}
If $k<n$ and $M$ is a Riemannian manifold with uniform two-sided bounds on sectional curvature (independent of the Gromov-Hausdorff distance $\mu$), then the Euler characteristic of $M$ vanishes.
This follows from \cite[Proposition 1.5]{CG:i} and \cite[Theorem 0.1]{CG:ii}.
\end{rem}

\begin{rem}
Gromov's Betti number theorem \cite{G:bet} and its generalization to Alexandrov spaces (\cite{LS}, \cite{Y:ess}) tell us that $\chi(M)$ is uniformly bounded in terms of dimension, a lower curvature bound, and an upper diameter bound.
As we will see later, this also holds for $\chi(F_E)$ in the above formula.
In particular, if $M_j$ is a sequence of Alexandrov spaces of fixed dimension that converges to $X$, then after passing to a subsequence, one can assume that all the Euler characteristics appearing in the formula of Theorem \ref{thm:main} are independent of the sequence.
This demonstrates Conjecture \ref{conj:A}(2b).
See Appendix \ref{sec:A} for details.
\end{rem}

This paper is actually a continuation (or addendum) of \cite{F:good}, where the author proved Theorem \ref{thm:main} in the case when $X$ has no proper extremal subsets.
Indeed the geometric ingredients needed for the proof have already been obtained in the previous paper, and what we do here is a purely topological argument.
In \cite{F:good} the author combined Perelman's Serre fibration theorem \cite{Per:col} with the good covering method developed by Mitsuishi-Yamaguchi \cite{MY:good}.
This makes it possible to compute the cohomology of $M$ via a spectral sequence associated with a presheaf on a good covering of $X$.
This presheaf is in a sense constant on the strata of $X$, and hence the Euler characteristic splits on each stratum as in the case of a regular fibration.
This reduces the computation of the Euler characteristic of $M$ to that of $X$ in terms of its strata.
The latter is done with the help of another result of the author \cite{F:reg}, which showed that each stratum has a deformation retract neighborhood.

\step{Organization}
In Section \ref{sec:pre} we fix notation and recall some results from \cite{F:good} and \cite{F:reg}.
In Section \ref{sec:prf} we prove Theorem \ref{thm:main}.
In Appendix \ref{sec:A}, to complete the proof of Alesker's conjecture, we construct the integer valued function $F$ on the strata of the limit space asserted in Conjecture \ref{conj:A}(2b).
As already mentioned, this is an unpublished result of Anton Petrunin.

\step{Acknowledgment}
I would like to thank Professor Semyon Alesker for his interest in my work.

\section{Preliminaries}\label{sec:pre}

The reader is assumed to be familiar with the theory of Alexandrov spaces and extremal subsets (see \cite[Section 3]{F:good} for a brief summary).
Here we fix notation and recall some results from \cite{F:good} and \cite{F:reg}.

Let $X$ and $M$ be Alexandrov spaces as in Problem \ref{prob:main} that are $\mu$-close in the Gromov-Hausdorff distance.
Note that $X$ is a fixed space whereas $M$ is a variable space depending on the choice of $\mu$.
We fix a $\mu$-approximation between $X$ and $M$ and use the hat symbol $\hat\ $ to indicate lifts from $X$ to $M$.
For example, for $p\in X$, we denote by $\hat p\in M$ a point that is $\mu$-close to $p$ under this approximation.

Recall that the distance function from $p\in X$ is called \textit{regular} at $x\in X\setminus\{p\}$ if $\tilde\angle pxy>\pi/2$ for some $y\in X$, where $\tilde\angle$ denotes the comparison angle.
Note that there exists a neighborhood of $p$ on which the distance function from $p$ is regular except at $p$.
We denote by $\bar B(p,r)$ the closed $r$-ball around $p$.

\begin{dfn}\label{dfn:fbr}
Let $p\in X$ and let $r>0$ be such that the distance function from $p$ is regular on $\bar B(p,r)\setminus\{p\}$.
We call such a pair $(p,r)$ \textit{fiber data} on $X$.
We say that $\bar B(\hat p,r)$ is a \textit{regular fiber} over $p$ in $M$, provided $\mu\ll r$.
\end{dfn}

\begin{rem}
The choice of $\mu$ will be determined in the proof of each statement.
\end{rem}

\begin{rem}
Fiber data is defined on a fixed space $X$, whereas its regular fiber depends on a variable space $M$.
\end{rem}

\begin{rem}\label{rem:fbr}
One can also use the open ball $B(\hat p,r)$ instead of the closed ball to define a regular fiber.
Indeed, they are homotopy equivalent by Perelman's fibration theorem (\cite[Theorem 1.4.1]{Per:alex}, \cite[Theorem 1.4(B)]{Per:mor}).
\end{rem}

Let $E$ be an extremal subset of $X$ (see \cite[Section 3C]{F:good} for extremal subsets).
By definition, $E$ is closed under the gradient flow of any semiconcave function.
The family of extremal subsets is closed under taking union, intersection, and closure of difference.
The number of extremal subsets in $X$ is finite.
In fact, it is uniformly bounded in terms of dimension, a lower curvature bound, and an upper diameter bound (\cite[Theorem 4.5]{A:conj}, \cite[Theorem 1.1(1)]{F:uni}).

Recall that $E$ is called \textit{primitive} if it cannot be represented as a union of two proper extremal subsets.
For a primitive extremal subset $E$, its \textit{main part} $\mathring E$ is defined as the relative complement of all proper extremal subsets in $E$.
Note that $p\in\mathring E$ if and only if $E$ is the minimal extremal subset containing $p$, that is, the intersection of all extremal subsets containing $p$, which is primitive.
It is known that $\mathring E$ is a topological manifold.
Therefore the main parts of all primitive extremal subsets define a stratification of $X$ (\cite[Section 3.8]{PP:ext}).

The next lemma shows that regular fibers are constant on each stratum of $X$.

\begin{lem}[{\cite[Lemma 5.11]{F:good}}]\label{lem:fbr}
Let $E$ be a primitive extremal subset of $X$.
Let $(p,r)$ and $(q,s)$ be fiber data on $X$ such that $p,q\in\mathring E$.
If $\mu$ is small enough, then the regular fibers $\bar B(p,r)$ and $\bar B(q,s)$ are homotopy equivalent.
\end{lem}

Note that the choice of $\mu$ depends on $r$ and $s$, and hence on $p$ and $q$.
This lemma allows us to define the notion of a \textit{regular fiber} over $E$ up to homotopy equivalence, which appeared in Theorem \ref{thm:main}.
Strictly speaking, we choose some fiber data $(p,r)$ such that $p\in\mathring E$ for each primitive extremal subset $E$ of $X$, and then take $\mu$ to be sufficiently small so that the conclusion of the theorem holds.

Let $\{U_\alpha\}_{\alpha=1}^N$ be a good covering of $X$ (see \cite[Section 3D]{F:good} for good coverings).
Recall that each $U_\alpha$ is a superlevel set of a strictly concave function constructed from distance functions.
For any nonempty subset $A\subset\{1,\dots,N\}$, we denote by $U_A$ the intersection of all $U_\alpha$ such that $\alpha\in A$.
We also denote by $I$ the set of $A$ such that $U_A$ is nonempty.
For any $A\in I$, $U_A$ is contractible by gradient flows of semiconcave functions.
In particular it is a good cover in the topological sense.
Since $U_A$ is defined by distance functions, it can be lifted to $M$, denoted by $\hat U_A$.

Let $E$ be an extremal subset of $X$ (not necessarily primitive).
We denote by $I_E$ the set of $A\in I$ such that $U_A$ intersects $E$.
Moreover, if $E$ is primitive, we denote by $\mathring I_E$ the set of $A\in I_E$ such that $U_A$ does not intersect $E\setminus\mathring E$.
Since $F=E\setminus\mathring E$ is also extremal, we see that $A\in\mathring I_E$ if and only if $E$ is the minimal extremal subset intersecting $U_A$.
Clearly $I_E$ is the disjoint union of $\mathring I_E$ and $I_F$.

For any extremal subset $E$ of $X$, the restricted cover $\{U_\alpha\cap E\}_{\alpha\in I_E}$ is also good in the topological sense.
This follows from the fact that any gradient flow preserves extremal subsets.
Note that $A=\{\alpha_1,\dots,\alpha_k\}\in I_E$ if and only if $A\in I$ and $\{\alpha_i\}\in I_E$ for any $1\le i\le k$ (see the paragraph after the proof of \cite[Lemma 5.14]{F:good}).
In other words, $I_E$ defines a full subcomplex of $I$ in the associated nerve, but we will not use this fact (cf.\ Proposition \ref{prop:ret}).

The following proposition is the key ingredient in the proof of Theorem \ref{thm:main}.

\begin{prop}[{\cite[Proposition 5.13]{F:good}}]\label{prop:good}
Let $E$ be a primitive extremal subset of $X$.
If $A\in \mathring I_E$, then $\hat U_A$ has the homotopy type of a regular fiber over $E$.
More precisely, for any fiber data $(p,r)$ such that $p\in\mathring E$, if $\mu$ is small enough, $\hat U_A$ is homotopy equivalent to $\bar B(\hat p,r)$.
\end{prop}

We also use the following fact from \cite{F:reg}.

\begin{prop}[{\cite[Theorem 1.5]{F:reg}}]\label{prop:ret}
Let $E$ and $F$ be extremal subsets of $X$ such that $F\subset E$.
Then any sufficiently small metric neighborhood of $F$ in $E$ admits a deformation retraction to $F$.
\end{prop}

The case $E=X$ was proved in \cite[Theorem 1.5]{F:reg}.
Since the deformation retraction of this theorem was given by a gradient flow which preserves extremal subsets, the general case follows by restricting it to $E$.

\begin{rem}\label{rem:ext}
In Lemma \ref{lem:fbr} and Proposition \ref{prop:good}, the homotopy equivalences can be chosen to preserve extremal subsets of $M$.
For example, in Lemma \ref{lem:fbr}, if $G$ is an extremal subset of $M$, then $\bar B(p,r)\cap G$ is homotopy equivalent to $\bar B(q,s)\cap G$ by the restriction of the original homotopy equivalences.
This is because the fibration theorem and the gradient flows used in the proofs preserve extremal subsets (see \cite[Section 9]{K:stab} for the fibration theorem).
This observation allows us to generalize Theorem \ref{thm:main} to each extremal subset of $M$ (see Remark \ref{rem:ext2}).
\end{rem}

\section{Proof}\label{sec:prf}

In this section we prove Theorem \ref{thm:main}.
We refer to \cite[Chapter III]{BT} for the basic theory of spectral sequences.

\begin{proof}[Proof of Theorem \ref{thm:main}]
Let $X$ and $M$ be as in Problem \ref{prob:main}.
Let $\{U_\alpha\}_{\alpha=1}^N$ be a good covering of $X$.
Let $I^p$ be the set of $A\in I$ with cardinality $p+1$ (see Section \ref{sec:pre} for the definition of $I$).
We denote by $S^q(U)$ the module of singular $q$-cochains on $U$, where the coefficient field is fixed and omitted.

We consider the Mayer-Vietoris double complex $\{\prod_{A\in I^p}S^q(\hat U_A)\}_{p,q\ge0}$, that is, the rows have the \v Cech coboundary operator (alternating sum of restrictions) and the columns have the singular coboundary operator.
The spectral sequence associated with this double complex converges to the cohomology of $M$.

Let $\mathcal E$ denote the set of primitive extremal subsets in $X$.
For each $E\in\mathcal E$, we take fiber data $(p_E,r_E)$ such that $p_E\in\mathring E$ and denote its regular fiber by $F_E$.
Note that $I$ is the disjoint union of $\mathring I_E$ for all $E\in\mathcal E$ (see Section \ref{sec:pre} for the definition of $\mathring I_E$).
Let $\mathring I^p_E$ be the intersection of $I^p$ and $\mathring I_E$.
Since taking homology does not change the Euler characteristic, by looking at the $E_1$ term of the spectral sequence, we get
\begin{align*}
\chi(M)&=\sum_{p,q\ge 0}\sum_{A\in I^p}(-1)^{p+q}\dim H^q(\hat U_A)\\
&=\sum_{p,q\ge0}\sum_{E\in\mathcal E}\sum_{A\in\mathring I_E^p}(-1)^{p+q}\dim H^q(F_E)\tag{$\because$ Proposition \ref{prop:good}}\\
&=\sum_{E\in\mathcal E}\left(\sum_{q\ge 0}(-1)^q\dim H^q(F_E)\cdot\sum_{p\ge0}\sum_{A\in\mathring I_E^p}(-1)^p\right)\\
&=\sum_{E\in\mathcal E}\left(\chi(F_E)\cdot\sum_{p\ge0}\sum_{A\in\mathring I_E^p}(-1)^p\right).
\end{align*}
Hence it remains to show that
\[\sum_{p\ge 0}\sum_{A\in\mathring I_E^p}(-1)^p=\chi_c(\mathring E).\]

The left hand side $\sum_{p\ge0}\sum_{A\in\mathring I_E^p}(-1)^p$ is the Euler characteristic of the \v Cech complex $C_{\mathring E}=\{\prod_{A\in\mathring I^p_E}C(U_A)\}_{p\ge0}$, where $C(U)$ denotes the module of constant functions on $U$.
Let $F=E\setminus\mathring E$, which is also an extremal subset.
Consider the \v Cech complexes $C_E=\{\prod_{A\in I^p_E}C(U_A)\}_{p\ge0}$ and $C_F=\{\prod_{A\in I^p_F}C(U_A)\}_{p\ge0}$.
Note that $C_{\mathring E}$ is the quotient of $C_E$ by $C_F$.
Recall that $\{U_\alpha\cap E\}_{\alpha\in I_E}$ and $\{U_\alpha\cap F\}_{\alpha\in I_F}$ are topological good covers of $E$ and $F$, respectively.
Hence the cohomologies of $C_E$ and $C_F$ are isomorphic to the singular cohomology groups $H^\ast(E)$ and $H^\ast(F)$, respectively, via spectral sequences.
Moreover, since the contractions of the good covers are given by the same gradient flows, these natural isomorphisms commute with the horizontal exact sequences in the following diagram:
\[
\begin{CD}
@>>>	H^*(C_{\mathring E})	@>>>	H^*(C_E)	@>>>	H^*(C_F)	@>>>	\\
@.		@VVV						@VVV				@VVV				\\
@>>>	H^*(E,F)				@>>>	H^*(E)		@>>>	H^*(F)		@>>>	.
\end{CD}
\]
Therefore $H^*(C_{\mathring E})$ is isomorphic to $H^\ast(E,F)$ (see also the following Remark \ref{rem:lip}).
By Proposition \ref{prop:ret}, $H^\ast(E,F)$ is isomorphic to the cohomology group $H^\ast_c(\mathring E)$ with compact support.
This completes the proof.
\end{proof}

\begin{rem}\label{rem:lip}
More generally, the pair $(E,F)$ is homotopy equivalent to the geometric realization of the pair of the nerves of the restricted good covers $(\{U_\alpha\cap E\}_{\alpha\in I_E},\{U_\alpha\cap F\}_{\alpha\in I_F})$.
This also yields the above isomorphism.
See \cite[Complement 8.4]{FMY} for the proof.
\end{rem}

\begin{rem}
As shown in the above proof, the Betti numbers with compact support of $\mathring E$ are finite.
Moreover, they are uniformly bounded in terms of dimension, a lower curvature bound, and an upper diameter bound, since the Betti numbers of any extremal subset are uniformly bounded in terms of these constants (\cite[Theorem 1.1(2)]{F:uni}).
\end{rem}

\begin{rem}
The Betti numbers without compact support of $\mathring E$ are also uniformly bounded.
Indeed, since the distance function from $F$ is regular near it (\cite[Lemma 3.1(2)]{PP:ext}), the complement of a small metric neighborhood of $F$ in $E$ is a deformation retract of $\mathring E$ by the gradient flow.
This, together with the proof of the uniform boundedness in \cite[Section 5]{F:uni} (cf.\ \cite[Theorem 5.2]{Y:ess}), implies the claim.
\end{rem}

\begin{rem}\label{rem:ext2}
In view of Remark \ref{rem:ext}, if $G$ be an extremal subset of $M$, then
\[\chi(G)=\sum_{E\in\mathcal E}\chi_c(\mathring E)\cdot\chi(F_E\cap G),\]
where we regard $\chi(F_E\cap G)=0$ if $F_E\cap G=\emptyset$.
The proof is the same as above.
\end{rem}

\appendix

\section{}\label{sec:A}

In this appendix, to complete the proof of Alesker's conjecture, we construct the integer valued function $F$ on the strata of the limit space asserted in Conjecture \ref{conj:A}(2b).
This is an unpublished result of Petrunin stated in \cite[Theorem 4.6]{A:conj} without proof.
As before, we fix and omit the coefficient field of singular cohomology.

\begin{thm}[Petrunin]
Let $M_j$ be a sequence of $n$-dimensional Alexandrov spaces converging to a compact Alexandrov space $X$.
Let $\mathcal E$ denote the set of primitive extremal subsets in $X$.
Passing to a subsequence, one can construct a function $F:\mathcal E\to\mathbb Z$ satisfying the following properties:
\begin{enumerate}
\item Let $(p,r)$ be fiber data on $X$ (see Definition \ref{dfn:fbr}) and let $E$ be a primitive extremal subset of $X$ such that $p\in\mathring E$.
Then we have
\[F(E)=\lim_{j\to\infty}\chi(\bar B(p_j,r))\]
for any sequence $p_j\in M_j$ converging to $p$.
\item Assume that there is another way of convergence $M_j\overset{\mathrm{GH}}\longrightarrow X$.
Passing to a subsequence again, one can also construct $F':\mathcal E\to\mathbb Z$ for this convergence.
Then there exists an isometry $\iota:X\to X$ such that the induced bijection $\iota_\ast:\mathcal E\to\mathcal E$ satisfies
\[F=F'\circ\iota_\ast.\]
\end{enumerate}
\end{thm}

\begin{rem}
In (1) the sufficiently large index $j$ at which $\chi(\bar B(p_j,r))$ becomes constant depends on $r$ and hence on $p$.
Note that $r$ is smaller than the distance from $p$ to the extremal subset $E\setminus\mathring E$.
\end{rem}

\begin{rem}
As in Remark \ref{rem:fbr}, one can use open balls instead of closed balls in the above statement and the following proof.
\end{rem}

\begin{rem}
Alesker mentioned that the function $F$ will be independent of the choice of the coefficient field.
\end{rem}

\begin{proof}
For each $E\in\mathcal E$, we take fiber data $(p_E,r_E)$ such that $p_E\in\mathring E$ as in the proof of the main theorem.
Let $p_E^j\in M_j$ be a sequence converging to $p_E$.
By the generalization of Gromov's Betti number theorem to Alexandrov spaces (\cite{LS}, \cite{Y:ess}), the total Betti number of the regular fiber $\bar B(p_E^j,r_E)$ is uniformly bounded independent of $j$.
Indeed, by \cite[Theorem 5.2]{Y:ess}, the rank of the inclusion homomorphism $H^\ast(\bar B(p_E^j,r_E/2))\to H^\ast(\bar B(p_E^j,r_E))$ is uniformly bounded in terms of dimension and a lower curvature bound (note that the bound is independent of $r_E$ by rescaling).
Since the distance function from $p_E^j$ is regular on $\bar B(p_E^j,r_E)\setminus B(p_E^j,r_E/2)$ for sufficiently large $j$, the inclusion $\bar B(p_E^j,r_E/2)\hookrightarrow\bar B(p_E^j,r_E)$ is a homotopy equivalence by Perelman's fibration theorem (\cite[Theorem 1.4.1]{Per:alex}, \cite[Theorem 1.4(B)]{Per:mor}).
Hence the total Betti number of this regular fiber is uniformly bounded independent of $j$.
Since the number of primitive extremal subsets is finite, after passing to a subsequence, one can assume that $\chi(\bar B(p_E^j,r_E))$ is constant independent of $j$ for any $E\in\mathcal E$.
We define this value to be $F(E)$.

Let us prove (1).
Let $(p,r)$ be fiber data on $X$ such that $p\in\mathring E$ and assume that $p_j\in M_j$ converges to $p$.
By Lemma \ref{lem:fbr}, $\bar B(p_j,r)$ is homotopy equivalent to $\bar B(p_E^j,r_E)$ for sufficiently large $j$ (depending on $r$).
This implies the first claim.

Let us prove (2).
We define $F'$ for the second convergence in the same way as above.
The isometry $\iota$ is constructed as follows.
For any $p\in X$, take a sequence $p_j\in M_j$ converging to $p$ with respect to the first convergence.
Passing to a subsequence, one can assume that $p_j$ also converges in the second convergence.
This limit point should be $\iota(p)$.
By a standard diagonal argument, one can define $\iota$ on some countable dense subset of $X$.
Clearly $\iota$ preserves distance.
Hence there exists a unique distance-preserving extension of $\iota$ onto $X$.
Since $X$ is compact, it is surjective (see \cite[Theorem 1.6.14]{BBI}).

It is easy to see that $F=F'\circ\iota_\ast$.
Indeed, if $(p,r)$ is fiber data with $p\in\mathring E$, then $(\iota(p),r)$ is also fiber data and $\iota(p)$ is contained in the main part of $\iota_\ast(E)$.
Moreover, if $p_j\in M_j$ is a sequence converging to $p$ in the first convergence, then it also converges to $\iota(p)$ in the second convergence.
These together with the property (1) imply the second claim.
\end{proof}

\begin{rem}\label{rem:ess}
As shown in the first paragraph of the above proof, the total Betti number of a regular fiber is uniformly bounded in terms of dimension and a lower curvature bound.
\end{rem}

\begin{rem}
Although we only consider the Euler characteristics of regular fibers in the above theorem, one can actually consider their Betti numbers.
Furthermore, one can consider the Betti numbers of the intersections of regular fibers with extremal subsets of $M_j$, as in \cite[Theorem 4.6(1)]{A:conj}.
This is because the numbers of extremal subsets in $M_j$ are uniformly bounded (\cite[Theorem 4.5]{A:conj}, \cite[Theorem 1.1(1)]{F:uni}) and Gromov's Betti number theorem also holds for extremal subsets (\cite[Theorem 1.1(2)]{F:uni}).
\end{rem}

\end{document}